\documentclass[11pt]{amsart}

 \theoremstyle{plain}

\numberwithin{equation}{section}

\newtheorem{Theorem}{Theorem}

\numberwithin{Theorem}{section}

\newtheorem{Lemma}[Theorem]{Lemma}

\title[Global existence of classical solutions]{Global existence of classical solutions for\\ a nonlocal
one dimensional parabolic \\free boundary problem}

\author{Rossitza Semerdjieva}

\address{Institute of Mathematics and Informatics, Bulgarian Academy
of Sciences, Acad. G. Bonchev Str., Bl. 8, 1113 Sofia, Bulgaria}
\email{rsemerdjieva@yahoo.com, rossitza@math.bas.bg}

\begin{document}

\begin{abstract}
In this paper we study one dimensional parabolic free boundary value
problem with a nonlocal (integro-differential) condition on the free
boundary.  We establish global existence--uniqueness of classical
solutions assuming that the initial-boundary data are sufficiently
smooth and satisfy some compatibility conditions. Our approach is
based on analysis of an equivalent system of nonlinear integral
equations.
\vspace{2mm} \\

\vspace{2mm}

\end{abstract}

\maketitle

\section{Introduction}

In this paper we consider the following free boundary value problem.

{\bf Problem $\bf P.$}  Find $s(t)>0$ and $u(x,t) $ such that
\begin{equation}
\label{1}  u_t = u_{xx}- \lambda u, \quad  \lambda =const>0, \quad 0
< x < s(t), \;\; t>0;
\end{equation}
\begin{equation}
\label{2}  s^\prime (t) =  \int_0^{s(t)} \left (u(x,t) -
\tilde{\sigma}
 \right ) dx,   \quad t > 0, \quad
 \tilde{\sigma}=const >0;
\end{equation}
\begin{equation}
\label{3} u (0, t) = f(t),  \quad    f(t)
> 0, \quad  t\geq 0;
\end{equation}
\begin{equation}
\label{4}  u(x,0) = \varphi (x),  \quad \varphi (x)
>0, \;\; x\in [0,b], \;\; s(0)=b >0, \quad  \varphi (0) = f(0);
\end{equation}
\begin{equation}
\label{5} u_x (s(t), t) = 0, \quad t>0.
\end{equation}
This is one-dimensional free boundary problem with unknown boundary
$x= s(t). $ Notice that (\ref{2}) is a nonlocal condition on the free
boundary, and (\ref{3})--(\ref{5}) are mixed type boundary conditions
for the parabolic equation (\ref{1}).

The aim of this paper is to investigate the existence and uniqueness
of classical solutions of Problem $P$. To this end we reduce the
problem to a system of nonlinear integral equations and analyze its
local solvability. The same approach has been used in many papers on
the one-dimensional Stefan problem and its variations in order to
prove existence-uniqueness results -- see, for instance, \cite[Ch.
8]{F} and \cite{Rub} and the bibliography given there. In the context
of tumor models, existence-uniqueness results for free boundary
problems similar to Problem $P$ are obtained in \cite[Theorem
3.1]{FR99} and \cite[Theorem 2.1]{CF01} by the same method. However,
the presence of mixed type boundary conditions in Problem $P$ brings
to some additional difficulties, and as far as we know this case has
not been studied yet.

Our main result is the following.

\begin{Theorem}
\label{thm1} Suppose that
\begin{equation}
\label{6} f(t) \in C^1([0,\infty)), \quad f(t)> 0 \;\; \text{for}\;
\; t \geq 0, \quad \varphi (x) \in C^2([0,b]),
\end{equation}
$$\varphi (x)>0 \;\;
\text{for} \; x\in [0,b], \quad f(0) =\varphi (0), \quad f^\prime (0)
= \varphi^{\prime \prime} (0) - \lambda \varphi (0), \quad
\varphi^\prime (b) =0.
$$
Then there exists a unique pair of functions $u (x,t)$ and $s(t) $
such that

(i) \hspace{1mm} $u (x,t) $ is defined, continuous and has continuous
partial derivatives $u_x, \, u_t, \, u_{xx} $ in the domain $\{(x,t):
\;  0 \leq x \leq s(t), \, t \geq 0 \};$

(ii) \hspace{1mm} $s(t) \in C^1 ([0,\infty)); $

(iii) \hspace{1mm} (\ref{1})--(\ref{5}) hold.
\end{Theorem}

In order to prove this theorem we introduce an auxiliary free
boundary value problem (see Problem~$\tilde{P}$ in Section~3) and
analyze local and global in time solvability of the resulting pair of
free boundary value problems. In Lemma~\ref{lem3} it is shown that
every solution of the main Problem $P$ ((\ref{1})--(\ref{5}))
generates a solution of the auxiliary problem and vice versa.
Existence and uniqueness of local solutions of the auxiliary problem
are proved  by deriving and studying an equivalent system of
nonlinear integral equations (see Lemma~\ref{lem4} and
Lemma~\ref{lem5}). In Lemma~\ref{lem7} we obtain a priori estimates
for the local solutions of the auxiliary problem by applying an
appropriate maximum principle (see Lemma~\ref{lem1} and
Lemma~\ref{lem2}) to the solutions of the main problem. Finally, we
prove existence of global solutions for both the main and the
auxiliary problems  by using the corresponding a priori estimates
obtained in Lemma~\ref{lem7}. Some of these results are announced
without proofs in \cite{R2}.

\section{Preliminary results}

Throughout the paper we assume that the functions $f $ and $\varphi $
satisfy the conditions~(\ref{6}). \vspace{2mm}

{\bf Definition 1.} We say that a pair of functions $(u (x,t), s(t))$
is a solution of Problem~$P$ for $t\in [0,T), \; T \leq \infty, $ if

(i) $u (x,t) $ is defined, continuous and has continuous partial
derivatives $u_x, \, u_t, \, u_{xx} $ in the domain $D_T =\{(x,t): \;
0 \leq x \leq s(t), \; 0\leq t<T \}; $

(ii) the equation (\ref{1}) is satisfied for $t<T;$

(iii) $s(t) \in C^1 ([0,T)); $

(iv) the conditions (\ref{2})--(\ref{5}) hold for $t\in
[0,T).$\vspace{2mm}

\begin{Lemma}(Maximum Principle)
\label{lem1} Let $ \lambda = const > 0, \; s(t) \in C^1 ( [0,T]),
 $ and let  $u(x,t) $ be defined and continuous in the domain
$\overline{D}_T = \{(x,t): \; 0 \leq x \leq s(t), \, 0 \leq t \leq
T\},$ have continuous partial derivatives $u_x, u_{xx} $ for $ 0 < x
\leq s(t), \, 0 < t \leq T, $  and have a continuous partial
derivative $u_t$ for $\, 0 < x < s(t), \, 0 < t < T. $ Suppose that
\begin{equation}
\label{1.1} - \frac{\partial u}{\partial t} + \frac{\partial^2
u}{\partial x^2} - \lambda u \geq 0 \quad \text{for} \; \; 0 < x <
s(t), \; 0 < t < T,
\end{equation}
and
\begin{equation}
\label{1.2} \frac{\partial u}{\partial x} (s(t),t) \leq 0 \quad
\text{for} \;\; 0 \leq t \leq T.
\end{equation}
If $M := \max \{u(x,t):  (x,t) \in \overline{D}_T \} > 0,$ then
$u(x,t)$ attains its maximum only on the union of the segments
$\{(0,t):  t \in [0,T]\}$ and $\{(x,0):  x\in [0, s(0)]\}.$
\end{Lemma}

\begin{proof} To the contrary, assume that  $u(x,t) $ attains its maximum $M>0 $ at a point $(x_0,
t_0) $ with $x_0 >0,  \, t_0 >0. $

(a) Suppose $ x_0 <s(t_0). $ Obviously,  $ u_x (x_0,t_0) =0. $  On
the other hand, in view of (\ref{1.1}), we have
$$  0 \leq \frac{1}{h} \left (u(x_0,t_0) -u(x_0, t_0 -h)\right ) = u_t (x_0, t_h) \leq
u_{xx} (x_0, t_h) -\lambda u (x_0, t_h),
$$
with $ 0<h <t_0 $ and   $ t_0-h < t_h < t_0. $ Therefore, letting $
h\to 0,$ we get $ u_{xx} (x_0, t_0) \geq \lambda M >0. $ But then the
function of one variable $ u(x,t_0) $ has a strict local minimum at
$x=x_0, $ which is impossible.

(b) Assume that $x_0 = s(t_0).$ Since $ s^\prime (t_0 ) $ exists,
there is a unit vector $\vec{\ell} = (\alpha, \beta) $ with $\alpha >
0 $ and $\beta > 0$ such that the segment $\{(s(t_0)- \alpha h, t_0 -
\beta h), \; h \in (0, \varepsilon) \}$ is in the interior of $D_T $
for sufficiently small $\varepsilon >0.$ Then
$$
0 \leq \frac{1}{h} \left (u(s(t_0), t_0) - u(s(t_0)- \alpha h, t_0 -
\beta h) \right )= \frac{\partial u}{\partial \vec{\ell}} (x_h, t_h),
$$
where $x_h = s(t_0) - \alpha \theta h, \; t_h = t_0 - \beta \theta
h,\; \theta = \theta (h) \in (0,1).$ Therefore, by (\ref{1.1}) it
follows
$$
0 \leq \alpha u_x (x_h, t_h) + \beta u_t (x_h, t_h) \leq \alpha u_x
(x_h, t_h) + \beta ( u_{xx} (x_h, t_h)- \lambda u(x_h, t_h)).
$$
Passing to a limit as $h \to 0$ we get
$$
0\leq \alpha u_x (s(t_0), t_0) + \beta ( u_{xx} (s(t_0), t_0)-
\lambda u(s(t_0), t_0)).
$$
By (\ref{1.2}) we have $u_x (s(t_0), t_0) \leq 0. $  Therefore, the
latter inequality yields $u_{xx} (s(t_0), t_0)\geq \lambda M
>0,$ which implies that $u_x (s(t_0) - h, t_0) < 0 $ for all
sufficiently small $h>0, $  say $h \in (0, \delta).$  But then it
follows that $u(s(t_0)- h, t_0) >u(s(t_0), t_0))=M $ for $h \in
(0,\delta),$ which  is impossible. This completes the proof.
\end{proof}

In view of (\ref{2}), the Maximum Principle  yields immediately the
following a priori estimates for $u(x,t) $ and  $s(t).$

\begin{Lemma}
\label{lem2} If a pair of functions $(u(x,t), s(t)) $ is a solution
of Problem~$P$ for $ 0 \leq  t < T <\infty, $ then
\begin{equation}
\label{1.7} 0 \leq u(x,t) \leq C_T, \quad  0 \leq x \leq s(t), \;\; 0
\leq t <T,
\end{equation}
\begin{equation}
\label{1.8} -\tilde{\sigma} s(t) \leq s^\prime (t) \leq
(C_T-\tilde{\sigma})s(t), \quad be^{-\tilde{\sigma}t} \leq s(t) \leq
be^{(C_T-\tilde{\sigma})t},
\end{equation}
where $b=s (0), \;   C_T = \max \left ( \sup_{[0,T)} f(t), \;
\sup_{[0,b]} \varphi (x) \right ).$
\end{Lemma}

\section{Auxiliary free boundary
problem}

Next we consider the following {\em auxiliary free boundary value
problem.}\vspace{2mm}

 {\bf Problem
${\bf \tilde{P}}.$} Find $s(t)>0$ and $\tilde{u}(x,t) $ such that
\begin{equation}
\label{11}  \tilde{u}_t = \tilde{u}_{xx} , \quad 0 < x < s(t), \;\;
t>0;
\end{equation}
\begin{equation}
\label{12}  s^\prime (t) =   (f(t)- \tilde{\sigma}) s(t) +
e^{-\lambda t} \int_0^{s(t)} \left ( \int_0^x \tilde{u}(\xi,t) d\xi
 \right ) dx,   \quad t>0;
\end{equation}
\begin{equation}
\label{13} \tilde{u}_x (0, t) = \tilde{f}(t), \quad \tilde{f}(t) \in
C ([0,\infty));
\end{equation}
\begin{equation}
\label{14}  \tilde{u}(x,0) = \tilde{\varphi} (x), \;
\tilde{\varphi}(x) \in C^1 ([0,b]), \; s(0)= b, \;
\tilde{\varphi}^\prime (0) = \tilde{f}(0), \; \tilde{\varphi} (b) =0;
\end{equation}
\begin{equation}
\label{15} \tilde{u}(s(t), t) = 0, \quad t>0.
\end{equation}

{\bf Definition 2.} We say that a pair of functions $\tilde{u} (x,t)
$ and $s(t)$ is a solution of Problem~$\tilde{P}$ for $t\in [0,T), \;
T \leq \infty,$ if

(i) $\tilde{u} (x,t) $ is defined and continuous in the domain $D_T
=\{(x,t): \; 0 \leq x \leq s(t), \; 0\leq t<T \},$ has continuous
partial derivative $\tilde{u}_x $ in  $D_T,$ and has continuous
partial derivatives $\tilde{u}_t, \; \tilde{u}_{xx} $ for $ 0 < x <
s(t), \; 0< t<T; $

(ii) the equation (\ref{11}) is satisfied for $t\in (0,T);$

(iii) $s(t) \in C^1 ([0,T)); $

(iv) the conditions (\ref{12})--(\ref{15}) hold for $t\in
[0,T).$\vspace{2mm}

The next lemma gives the relation between Problem $P$ and Problem
$\tilde{P}.$
\begin{Lemma}
\label{lem3} Let  $f(t)$ and $\varphi (x)$ satisfy (\ref{6}),
$\lambda =const
>0, $ and let
\begin{equation}
\label{17} \tilde{f}(t) = \frac{d}{dt} \left (e^{\lambda t}f (t)
\right ), \quad \tilde{\varphi} (x)= \varphi^\prime (x).
\end{equation}

(a) If a pair of functions $ u(x,t)$ and $ s(t)$ is a solution of
Problem~P for $t\in [0,T),$ then the pair of functions $\tilde{u}
(x,t) = e^{\lambda t}  u_x (x,t)$ and $s(t)$ is a solution of Problem
$\tilde{P}$ for $t\in [0,T).$

(b) If a pair of functions $\tilde{u} (x,t) $ and $s(t)$ is a
solution of Problem $\tilde{P}$ for $t\in [0,T),$ then the pair of
functions $( u(x,t), s(t))$ with
\begin{equation}
\label{16} u(x,t) = f(t) + e^{-\lambda t}\int_0^x \tilde{u} (\xi,t) d
\xi
\end{equation}
is a solution of  Problem $P$ for $t\in [0,T).$
\end{Lemma}

\begin{proof}
(a) Notice that $u(x,t)$  is a $C^\infty$-function in the interior of
the domain $D_T$ due to general smoothness theorems (see \cite[Ch.3,
Thm. 11]{F}, and Corollary 2 there). Therefore, the function $
\tilde{u} (x,t)  $ has continuous partial derivatives $\tilde{u}_t,
\, \tilde{u}_{xx} $ for $0 < x < s(t), \, 0 < t < T,$ and it
satisfies the equation $\tilde{u}_t = \tilde{u}_{xx} $ in that
domain.

Letting $x\to 0 $ in the equation (\ref{1}), we obtain $$ u_t (0,t) =
f^\prime (t) = u_{xx} (0,t) - \lambda f(t). $$ Thus, $\tilde{u}_x
(0,t) = u_{xx} (0,t)e^{\lambda t} = [f^\prime (t) + \lambda f(t)
]e^{\lambda t}, $ i.e., (\ref{13}) holds with $ \tilde{f} (t) =
\frac{d}{dt} \left ( f(t)e^{\lambda t} \right ).$ Now one can readily
verify  that the pair of functions $\tilde{u} (x,t)= e^{\lambda t}
u_x (x,t) $ and $s(t)$ is a solution of Problem $\tilde{P}$ for $t\in
[0,T).$

(b) We check first that the function $u(x,t) $ given in (\ref{16})
satisfies the equation (\ref{1}). By (\ref{16}), we have $$ u_{xx}
(x,t) = e^{-\lambda t} \tilde{u}_x (x,t) . $$ In order to find and
justify a formula for  $u_t $ we set
$$ u_n (x,t) = f(t) +e^{-\lambda t}\int_{1/n}^x \tilde{u} (\xi,t) d
\xi, \quad n=1,2, \ldots. $$ Then $ u_n (x,t) \to u(x,t) $  as $n\to
\infty $ for  $(x,t) \in D_T.$ In view of (\ref{11}), we have
$$
\frac{\partial }{\partial t} (u_n (x,t)) = f^\prime (t) -\lambda
e^{-\lambda t}\int_{1/n}^x \tilde{u} (\xi,t) d \xi + e^{-\lambda
t}\int_{1/n}^x \tilde{u}_{\xi \xi} (\xi,t) d \xi.
$$
Since $ \int_{1/n}^x \tilde{u}_{\xi \xi} (\xi,t) d \xi = \tilde{u}_x
(x,t)- \tilde{u}_x (1/n,t), $ we get  as $ n \to \infty$
$$
\frac{\partial }{\partial t} (u_n (x,t)) \to  f^\prime (t)  - \lambda
e^{-\lambda t}\int_{0}^x \tilde{u} (\xi,t) d \xi +e^{-\lambda t}
(\tilde{u}_x (x,t)- \tilde{u}_x (0,t))
$$
uniformly on any compact subinterval of $(0,T).$ Therefore, $u_t
(x,t)$ exists, and  using (\ref{16}) and (\ref{13}) we obtain
$$ u_t (x,t) = f^\prime (t)  - \lambda
(u(x,t) - f(t)) +e^{-\lambda t} (\tilde{u}_x (x,t)- \tilde{f} (t)).
$$
Since $ \tilde{f} (t) =\frac{d}{dt} \left ( f(t)e^{\lambda t} \right
)$ it follows that $ u(x,t) $ satisfies the equation (\ref{1}). Now
one can easily see that the pair of functions $(u(x,t), s(t)) $ is a
solution of  Problem $P$ for $t\in [0,T).$
\end{proof}

In view of Lemma \ref{lem3}, Theorem~\ref{thm1} will be proved if we
show that the following statement holds.
\begin{Theorem}
\label{thm2} Suppose that
\begin{equation}
\label{26} f(t) \in C^1([0,\infty)), \quad f(t)> 0 \;\; \text{for} \;
t \geq 0, \quad \tilde{f}(t) \in C([0,\infty)),
\end{equation}
$$\quad \tilde{\varphi}
(x) \in C^1([0,b]), \quad  \quad \tilde{f} (0) =
\tilde{\varphi}^{\prime} (0), \quad \tilde{\varphi} (b) =0.
$$
Then Problem $\tilde{P}$ has a unique solution for  $0\leq t
<\infty.$
\end{Theorem}

\section{System of integral equations}

In this section  Problem $\tilde{P}$ is transformed to an equivalent
problem of solving a system of nonlinear integral equations. We begin
with some preliminaries.

Consider the function
 \begin{equation}
 \label{e}
 K(x,t;\xi, \tau) = \frac{1}{2\sqrt{\pi}
\sqrt{t-\tau}} \exp \left ( {-\frac{(x-\xi)^2}{4(t-\tau)}}  \right ),
\quad \tau < t.
\end{equation}

We shall make use of the following elementary inequalities:
\begin{equation}
 \label{e1}
 \int_0^t \left |K_x(x,t;\xi, \tau) \right | d \tau = \frac{1}{\sqrt{\pi}}
\int_{\frac{|x-\xi|}{2\sqrt{t}}}^\infty e^{-z^2} dz \leq \frac{1}{2}
\end{equation}
(by performing the change of variable $z=
\frac{|x-\xi|}{2\sqrt{t-\tau}} \;$);
\begin{equation}
 \label{e2}
 \int_0^b  K(x,t;\xi, 0)  d \xi = \frac{1}{\sqrt{\pi}}
\int_{\frac{-x}{2\sqrt{t}}}^{\frac{b-x}{2\sqrt{t}}} e^{-z^2} dz \leq
1
\end{equation}
(by using  the change of variable $z= \frac{\xi - x}{2\sqrt{t}}\;);$
\begin{equation}
 \label{e3}
z^a e^{-\gamma z} \leq \left( \frac{a}{e \gamma}\right )^a \quad
\text{if} \quad z>0, \; a>0, \; \gamma >0.
\end{equation}

The next statement is a slight modification of Lemma 1 in
\cite[Ch.8]{F}), and its proof is the same.
\begin{Lemma}
\label{lemF} If $s(t) \in C^1 ([0,T]), \; g(t) \in C([0,T]) $ and
$0<t_0 <T,$ then
  $$\lim_{(x,t) \to (s(t_0),t_0)}  \int_0^t K_x (x,t;s(\tau),\tau)
g (\tau) d\tau = \frac{g(t_0)}{2}  +\int_0^t K_x
(s(t_0),t_0;s(\tau),\tau) g(\tau) d\tau,
$$
where in the limit we consider only points $(x,t) $ with $x<s(t).$
\end{Lemma}

Next we derive a system of integral equations related to Problem
$\tilde{P}.$ Let $N(x,t;\xi,\tau)$ be the Neumann function for the
half-plane $x>0,$  i.e.,
$$ N(x,t;\xi,\tau) = K(x,t;\xi, \tau)+ K(-x,t;\xi, \tau).$$ Suppose
that the pair of functions $(\tilde{u}(x,t), s(t))$ is a solution of
Problem $\tilde{P}$. For $t>0, $ we integrate the identity
\begin{equation}
\label{GI} \frac{\partial}{\partial \xi} \left ( N \frac{\partial
\tilde{u}}{\partial \xi} - \frac{\partial N}{\partial \xi} \tilde{u}
\right ) =\frac{\partial}{\partial \tau } (N\tilde{u}), \quad
\tilde{u} = \tilde{u}(\xi, \tau),
\end{equation}
 over the domain $
\varepsilon \leq \tau \leq t- \varepsilon, \quad \delta \leq \xi \leq
s(\tau) - \delta, \quad \varepsilon =const
>0, \; \delta = const >0, $  and pass to limits, first  as
$ \delta \to 0, $ and then as $\varepsilon \to 0. $ Since $N_\xi
(x,t;0,\tau)=0$ and
$$
\lim_{\varepsilon \to 0} \int_0^{s(t-\varepsilon )}
N(x,t;\xi,t-\varepsilon ) \tilde{u} (\xi,t-\varepsilon) d\xi =
\tilde{u}(x,t),
$$
it follows, in view of (\ref{13}) and (\ref{14}), that
\begin{equation}
\label{101} \tilde{u}(x,t)  = \sum_{\nu=1}^5 J_\nu (x,t),
\end{equation}
where
\begin{equation}
\label{102} J_1 (x,t) =\int_0^t N(x,t;s(\tau),\tau) v (\tau) d\tau
\quad \quad \text{with} \quad v (t) := \tilde{u}_x(s(t),t),
\end{equation}
$$
J_2 (x,t) =  -\int_0^t N (x,t;0,\tau) \tilde{f}(\tau) d\tau, \quad
J_3 (x,t) =\int_0^b N(x,t;\xi,0) \tilde{\varphi} (\xi) d\xi
$$
and
\begin{equation}
\label{103} J_4 (x,t) = -\int_0^t N_\xi (x,t;s(\tau),\tau)
\tilde{u}(s(\tau), \tau) d\tau,
\end{equation}
$$
J_5 (x,t) = \int_0^t s^\prime (\tau) N(x,t;s(\tau),\tau)
\tilde{u}(s(\tau), \tau) d\tau.
$$
The condition  (\ref{15}) implies $J_4 (x,t)=0, \; J_5 (x,t)=0. $
Thus, the following integral representation holds:
\begin{equation}
\label{110} \tilde{u}(x,t)  = J_1(x,t) + J_2 (x,t) + J_3 (x,t).
\end{equation}

Next, in order to obtain an integral equation for $ v(t) =
\tilde{u}_x (s(t),t), $  we differentiate (\ref{110}) with respect to
$x$ and pass to a limit as $x \to s(t)-0 $ in the resulting identity.
In view of Lemma \ref{lemF}, it follows that
  $$\lim_{x\to s(t)-0} \frac{\partial J_1}{\partial x} (x,t) =
  \int_0^t N_x (s(t),t;s(\tau),\tau) v
(\tau) d\tau + \frac{v(t)}{2}. $$  It is easy to see that
$$
\lim_{x\to s(t)-0} \frac{\partial J_2}{\partial x} (x,t) = - \int_0^t
N_x (s(t),t;0,\tau) \tilde{f}(\tau) d\tau.
$$
Now, consider the Green function for the half--plane $x>0$
$$G(x,t;\xi,\tau) = K(x,t;\xi, \tau)- K(-x,t;\xi, \tau). $$ Since $N_x
= - G_\xi, $ an integration by parts leads to
\begin{equation}
\label{111} \frac{\partial J_3}{\partial x} (x,t) = -\int_0^b G_\xi
(x,t;\xi,0) \tilde{\varphi} (\xi) d\xi=\int_0^b G (x,t;\xi,0)
\tilde{\varphi}^\prime (\xi) d\xi
\end{equation}
because $G(x,t;0,0) =0 $ and $\tilde{\varphi} (b)=0.$ Therefore, it
follows that
$$
\lim_{x\to s(t)-0} \frac{\partial J_3}{\partial x} (x,t) =\int_0^b G
(s(t),t;\xi,0) \tilde{\varphi}^\prime (\xi) d\xi.
$$
Hence, for $t>0 $ the function $v(t) $ satisfies the integral
equation
\begin{equation}
\label{112}  v(t)  = 2\int_0^t N_x (s(t),t;s(\tau),\tau)v (\tau)
d\tau
\end{equation}
$$
 -2\int_0^t N_x (s(t),t;0,\tau)
\tilde{f}(\tau) d\tau + 2\int_0^b G (s(t),t;\xi,0)
\tilde{\varphi}^\prime (\xi) d\xi.
$$

On the other hand, from (\ref{12}) and (\ref{110}) it follows
\begin{equation}
\label{113}  s^\prime (t)=  (f(t)e^{-\lambda t}-\tilde{\sigma}) s(t)
+ e^{-\lambda t} \left ( \int_0^{s(t)} \int_0^x \int_0^t
N(\xi,t;s(\tau),\tau) v (\tau)  d\tau d\xi dx \right.
\end{equation}
$$\left.
-\int_0^{s(t)} \int_0^x \int_0^t N (\xi,t;0,\tau) \tilde{f}(\tau)
d\tau d\xi dx + \int_0^{s(t)}\int_0^x\int_0^b N(\xi,t;\eta,0)
\tilde{\varphi} (\eta) d\eta d\xi dx \right ).
$$
The system of nonlinear integral equations (\ref{112}) and
(\ref{113}) considered with $s(t) =b+ \int_0^t s^\prime (\tau )
d\tau$ is equivalent to Problem~$\tilde{P}, $ i.e., the following
statement holds.

\begin{Lemma}
\label{lem4} Problem $\tilde{P}$ for $t<T$ is equivalent to the
problem of finding a pair of continuous functions $(v(t),\, s^\prime
(t))$ on $ [0,T) $ which satisfies for $t>0$ the system of nonlinear
integral equations (\ref{112}) and (\ref{113}) considered with $s(t)
=b+ \int_0^t s^\prime (\tau ) d\tau$.
\end{Lemma}

\begin{proof}
We have already proved that if a pair $(\tilde{u}(x,t), s(t))$ is a
solution of Problem~$\tilde{P}$ for $t<T$, then the pair of
continuous functions $v(t)= \tilde{u}_x (s(t),t) $ and $ s^\prime
(t), \; t\in [0,T), $ satisfies for $t>0$ the system of nonlinear
integral equations (\ref{112}) and (\ref{113}) considered with $s(t)
=b+ \int_0^t s^\prime (\tau ) d\tau.$

Conversely, suppose that a pair of continuous functions $v(t) $ and $
s^\prime (t), \; t\in [0,T),  $ satisfies for $t>0$ the system of
integral equations (\ref{112}) and (\ref{113}). Set
\begin{equation}
\label{120} \displaystyle \tilde{u}(x,t)= \begin{cases}
\sum_{\nu=1}^3 J_\nu (x,t) &
\quad  \text{for}   \;\; 0\leq x \leq s(t), \; 0<t<T, \\
\tilde{\varphi} (x) &\quad  \text{for}  \;\; 0\leq x \leq b, \quad \;
t=0,
\end{cases}
\end{equation}
where $J_\nu (x,t), \, \nu =1,2,3 $ are given by (\ref{102}) and
$s(t) = b+ \int_0^t s^\prime (\tau) d\tau.$ We shall prove that the
pair of functions $(\tilde{u}(x,t), s(t))$ form a solution of
Problem~$\tilde{P}$ for $t<T.$

First we show that the function $\tilde{u}(x,t)$ is continuous in the
domain $D_T.$ Indeed, since the integrands in $J_1 (x,t) $ and $J_2
(x,t)$ are dominated by a multiple of $(t-\tau)^{-1/2},$ we have
$$
\lim_{(x,t) \to (x_0,0)} J_\nu (x,t)= 0, \quad \nu =1,2, \quad  x_0
\in [0,b].
$$
So, it remains to show that $ J_3 (x,t) \to \tilde{\varphi} (x_0) $
as
 $(x,t) \to (x_0,0), \;x_0\in [0,b].$
Performing the change of variable $z = (\xi \pm x)/2\sqrt{t},$ we
obtain $J_3 (x,t)= J^1_3 (x,t)+J^2_3 (x,t), $ where
$$
J^1_3 (x,t)= \frac{1}{\sqrt{\pi}}
\int_{-\frac{x}{2\sqrt{t}}}^{\frac{b-x}{2\sqrt{t}}} e^{-z^2}
\tilde{\varphi} (x+ 2z\sqrt{t}) dz, \quad J^2_3 (x,t)=
\frac{1}{\sqrt{\pi}}
\int_{\frac{x}{2\sqrt{t}}}^{\frac{b+x}{2\sqrt{t}}} e^{-z^2}
\tilde{\varphi} (-x+ 2z\sqrt{t}) dz.
$$
Now, for every $x_0 \in (0,b), $ it follows that
$$
J^1_3 (x,t) \to \tilde{\varphi} (x_0) \quad \text{and} \quad  J^2_3
(x,t) \to 0 \quad \text{as} \;\; (x,t) \to (x_0,0)
$$
because $-x/2\sqrt{t} \to - \infty $ and  $(b\pm x)/2\sqrt{t} \to +
\infty. $

The corner points  $(0,0)$ and $(b,0)$ need a special consideration.
If $(x,t) \to (b,0),$ then $J^2_3 (x,t) \to 0$ by the same argument.
Since $\tilde{\varphi} (b)= 0, $ it follows that  $J^1_3 (x,t) \to 0
$ as well, so $J_3 (x,t) \to 0 = \tilde{\varphi} (b)$ as $(x,t) \to
(b,0).$

In order to show that $J_3 (x,t) \to \tilde{\varphi} (0)$ as $(x,t)
\to (0,0)$ we shall prove that $J_3 (x_n,t_n) \to \tilde{\varphi} (0)
$ for every sequence  $(x_n,t_n) \to (0,0).$ Since $J^1_3 (x,t) $ and
$J^2_3 (x,t) $  are bounded, the sequence $\{J_3 (x_n,t_n)\} $ is
bounded. Therefore, it is enough to show that every convergent
subsequence of the form $\{J_3 (x_{n_k},t_{n_k})\}$ has a limit equal
to $\tilde{\varphi} (0). $  We may assume that
$x_{n_k}/2\sqrt{t_{n_k}} \to a \in [0,\infty]$ (otherwise we may pass
to a subsequence of $(n_k)$). Then it follows that
$$
J_3 (x_{n_k},t_{n_k}) \to \frac{1}{\sqrt{\pi}} \int_{-a}^\infty
e^{-z^2} \tilde{\varphi} (0) dz +\frac{1}{\sqrt{\pi}} \int_{a}^\infty
e^{-z^2} \tilde{\varphi} (0) dz = \tilde{\varphi} (0).
$$
Thus, $J_3 (x,t) \to \tilde{\varphi} (0)$ as $(x,t) \to (0,0),$ which
completes the proof of continuity of $\tilde{u} (x,t) $ in the domain
$D_T. $

It is easy to see that each of the integrals $J_\nu (x,t) $ is a
$C^\infty $--function in the domain $ 0 < x < s(t), \; 0<t < T, $ and
satisfies the heat equation there. Thus, $\tilde{u} (x,t) $ satisfies
(\ref{11}) as well.

The functions $v(t) $ and $s^\prime (t)  $ are defined and continuous
on $[0,T) $ and satisfy the integral equations (\ref{112}) and
(\ref{113}) for $t\in (0,T). $
 Consider the limit  of the right-hand side of (\ref{112}) as
$t\to 0. $  It is easy to see that the first two integrals there
converge to zero. With $z = (\xi \pm s(t))/2\sqrt{t},$ the  integral
$\int_0^b G (s(t),t;\xi,0) \tilde{\varphi}^\prime (\xi) d\xi$ is
equal to
$$
\frac{1}{\sqrt{\pi}}
\int_{-\frac{s(t)}{2\sqrt{t}}}^{\frac{b-s(t)}{2\sqrt{t}}} e^{-z^2}
\tilde{\varphi}^\prime (s(t)+ 2z\sqrt{t}) dz - \frac{1}{\sqrt{\pi}}
\int_{\frac{s(t)}{2\sqrt{t}}}^{\frac{b+s(t)}{2\sqrt{t}}} e^{-z^2}
\tilde{\varphi}^\prime (-s(t)+ 2z\sqrt{t}) dz.
$$
As $t \to 0, $  the first integral in the above expression tends to
$\tilde{\varphi}(b)/2, $ while the second one tends to zero.
Therefore, by passing to limit as $t\to 0 $ in (\ref{112}) we obtain
\begin{equation}
\label{f} v(0) =  \tilde{\varphi}^\prime (b).
\end{equation}

Next we prove that  $\tilde{u}_x  (x,t) $ extends as a continuous
function on $ D_T. $ In view of (\ref{102}) and
(\ref{110})--(\ref{111}), we have
\begin{equation}
\label{f01} \tilde{u}_x  (x,t) = I_1 (x,t) + I_2 (x,t) + I_3 (x,t),
\quad (x,t) \in D_T^{\circ},
\end{equation}
where $D_T^{\circ} =\{(x,t): \; 0<x<s(t), \; 0<t<T\},$ and
\begin{equation}
\label{f02}  I_1 (x,t) = \int_0^t N_x (x,t;s(\tau),\tau) v(\tau)
d\tau,
\end{equation}
$$
I_2 (x,t) = - \int_0^t N_x (x,t;0,\tau) \tilde{f} (\tau) d\tau, \quad
I_3 (x,t) =  \int_0^b G (x,t;\xi,0) \tilde{\varphi}^\prime (\xi)
d\xi.
$$
We shall prove that
$$
(i)  \quad \tilde{u}_x  (x,t) \to \tilde{f} (t_0) \quad \text{as}
\quad (x,t) \to (0,t_0), \;\;(x,t) \in D_T^{\circ}, \;\;  0 < t_0< T;
$$
$$
(ii)  \quad \tilde{u}_x  (x,t) \to v(t_0) \quad \text{as} \quad (x,t)
\to (s(t_0),t_0), \;\; (x,t) \in D_T^{\circ},\;\;  0 < t_0< T;
$$
$$
(iii)  \quad \tilde{u}_x  (x,t) \to \tilde{\varphi}^\prime (x_0)
\quad \text{as} \quad (x,t) \to (x_0,0), \;\; (x,t) \in D_T^{\circ},
\;\;  0 \leq x_0 \leq b.
$$
Since the functions $\tilde{f} (t), \, \tilde{\varphi}^\prime (x), \,
v(t) $ are continuous and $\tilde{f} (0)= \tilde{\varphi}^\prime
(0),\;  \tilde{\varphi}^\prime (b)=v(0)$ (see (\ref{14}) and
(\ref{f})), the conditions $(i)-(iii)$ guarantee that $\tilde{u}_x $
extends as a continuous function on $D_T.$

First we prove (i). Taking into account that $N_x (0,t,\xi,\tau) = 0
$ and $G(0,t,\xi,\tau) = 0,$ it is easy to see that $I_\nu (x,t) \to
(0,t_0)$ for $ \nu=1,3 $ as $x\to +0, \, t\to t_0 \in (0,T).$ With
the change of variable $z= \frac{x}{2\sqrt{t-\tau}}$  the integral
$I_2 (x,t)$ becomes
$$
I_2 (x,t)= -\int_0^t N_x (x,t;0,\tau)\tilde{f} (\tau) d\tau =
 \frac{2}{\sqrt{\pi}}
\int_{\frac{x}{2\sqrt{t}}}^\infty e^{-z^2} \tilde{f} \left (t -
\frac{x^2}{4z^2}\right ) dz.
$$
As $ x \to +0, \; t\to t_0\in (0,T) $ the latter integral tends to
$\tilde{f} (t_0).$  Thus $(i)$ holds.

Next we prove (ii). One can easily see that $$I_\nu (x,t) \to I_\nu
(s(t_0),t_0)  \quad \text{as} \quad  (x,t) \to (s(t_0),t_0), \quad
\nu =2,3.$$ From Lemma~\ref{lemF} and (\ref{f02}) it follows that
$$
\lim_{(x,t) \to (s(t_0),t_0)} I_1 (x,t)= \frac{1}{2} v(t_0) + I_1
(s(t_0),t_0).
$$
Therefore, by (\ref{112}), we obtain
$$
\tilde{u}_x (x,t) \to \frac{1}{2} v(t_0) + \sum_{\nu=1}^3 I_\nu
(s(t_0),t_0) = v(t_0) \quad \text{as} \quad  (x,t) \to (s(t_0),t_0),
$$
i.e., $(ii)$ holds.

It is easy to verify $(iii)$ for $x_0 \in (0,b). $  However, it is
much more complicated to prove (iii) for $x_0 = 0 $ or $x_0 =b.$

Next we show that $(iii)$ holds for $x_0=b,$  i.e.,
\begin{equation}
\label{f1} \tilde{u}_x  (x,t) \to \tilde{\varphi}^\prime (b) \quad
\text{as} \quad (x,t) \to (b,0), \;\; (x,t) \in D_T^{\circ}.
\end{equation}
Let $\{(x_n, t_n)\}$ be an arbitrary sequence such that $(x_n, t_n)
\to (b,0), \; (x_n, t_n)\in D_T^{\circ}.$  In order to prove that
$\tilde{u}_x (x_n,t_n) \to \tilde{\varphi}^\prime (b)$ it is enough
to show that for every subsequence $\{(x_{n_k}, t_{n_k})\}$ there is
a sub-subsequence $\{(x_{n_{k_m}}, t_{n_{k_m}})\}$ (which we denote
for convenience by $\{(x_m, t_m)\} $) such that
$$
\tilde{u}_x  (x_m,t_m) \to \tilde{\varphi}^\prime (b) \quad \text{as}
\;\; m \to \infty.
$$
We may assume without loss of generality (otherwise one may pass to
an appropriate subsequence) that
\begin{equation}
\label{f2} \frac{b-x_m}{2 \sqrt{t_m}} \to \alpha \in [0,\infty]
\quad \text{as} \;\; m \to \infty.
\end{equation}
(Notice that  $x_m <  s(t_m) $  and
$$
\frac{b-s(t_m)}{2 \sqrt{t_m}} = \frac{s(0)-s(t_m)}{2 \sqrt{t_m}} \to
0 \quad \text{as} \;\; m \to \infty;
$$
therefore, every cluster point $\alpha $ in (\ref{f2}) is
nonnegative.)

In view of (\ref{f01}),  in order to evaluate  $\lim_{m\to \infty}
\tilde{u}_x (x_m, t_m) $ one needs to find  $\lim_{m\to \infty} I_\nu
(x_m, t_m), \; \nu=1,2,3. $  First, consider the case $\nu =1.$
 We have $ I_1 (x,t) =I_{1,1} (x,t) +I_{1,2}
(x,t),$ where
\begin{equation}
\label{f11} I_{1,1} (x,t)= \frac{1}{\sqrt{\pi}}\int_0^t
\frac{s(\tau)-x}{4(t-\tau)^{3/2}}
e^{-\frac{(x-s(\tau))^2}{4(t-\tau)}} v(\tau) d \tau,
\end{equation}
\begin{equation}
\label{f12}
 I_{1,2}(x,t)= -\frac{1}{\sqrt{\pi}}\int_0^t
\frac{s(\tau)+x}{4(t-\tau)^{3/2}}
e^{-\frac{(x+s(\tau))^2}{4(t-\tau)}} v(\tau) d \tau.
\end{equation}
One can easily see that $I_{1,2}(x,t) \to 0$ as $(x,t) \to (b,0)$
because $x+s(\tau) > b > 0 $  for $(x,\tau)$ close to $(b,0).$ On the
other hand,
$$
I_{1,1}(x_m,t_m) =
I^1_{1,1}(x_m,t_m)+I^2_{1,1}(x_m,t_m)+I^3_{1,1}(x_m,t_m),
$$
where
$$
I^1_{1,1}(x_m,t_m)= \frac{1}{\sqrt{\pi}}\int_0^{t_m}
\frac{s(\tau)-s(t_m)}{4(t_m-\tau)^{3/2}}
e^{-\frac{[x_m-s(\tau)]^2}{4(t_m-\tau)}} v(\tau) d \tau,
$$
$$
I^2_{1,1}(x_m,t_m)= \frac{1}{\sqrt{\pi}}\int_0^{t_m}
\frac{s(t_m)-x_m}{4(t_m-\tau)^{3/2}} \left [
e^{-\frac{[x_m-s(\tau)]^2}{4(t_m-\tau)}}-e^{-\frac{[x_m-s(t_m)]^2}{4(t_m-\tau)}}
 \right ]v(\tau) d \tau,
$$
$$
I^3_{1,1}(x_m,t_m)= \frac{1}{\sqrt{\pi}}\int_0^{t_m}
\frac{s(t_m)-x_m}{4(t_m-\tau)^{3/2}}
e^{-\frac{[x_m-s(t_m)]^2}{4(t_m-\tau)}} v(\tau)  d \tau.
$$

Since $s^\prime (t) $ is continuous,  the Mean Value Theorem implies
that $|s(t_m)- s(\tau)| \leq const \cdot |t_m - \tau |.$  Therefore,
the absolute value of the integrand of $I^1_{1,1}(x_m,t_m)$ does not
exceed $C/\sqrt{t_m - \tau},$ which leads to $I^1_{1,1}(x_m,t_m) \leq
C \sqrt{t_m} \to 0$  as $m \to \infty.$

Changing the variable in $I^3_{1,1}$ by $\frac{s(t_m) -
x_m}{2\sqrt{t_m - \tau}} = z, $ we obtain
$$
I^3_{1,1}(x_m,t_m)= \frac{1}{\sqrt{\pi}} \int_{\frac{s(t_m) -
x_m}{2\sqrt{t_m}}}^\infty e^{-z^2} v \left (t_m -\frac{(s(t_m) -
x_m)^2}{4z^2} \right ) dz \to v(0) \cdot \frac{1}{\sqrt{\pi}}
\int_\alpha^\infty e^{-z^2} dz.
$$

The expression in the square brackets in the integral $I^2_{1,1}$ can
be written as $ \exp \left (-\frac{(x_m-s(t_m))^2}{4(t_m-\tau)}
\right ) \left ( e^{g_m (\tau)}  - 1 \right ), $ where
$$
g_m (\tau) = \frac{s(t_m) - s(\tau)}{4(t_m - \tau)} \cdot ( s(t_m)
+s(\tau) - 2x_m )  \to 0  \quad \text{as} \;\; m \to \infty
$$
uniformly for $ \tau \in [0,t_m].$ Therefore, the same change of
variable as in $I^3_{1,1}$ shows that $I^2_{1,1}(x_m,t_m) \to 0 \, $
as $\, m\to \infty.$

Hence, we obtain
\begin{equation}
\label{f4} I_1 (x_m, t_m) \to v(0) \cdot \frac{1}{\sqrt{\pi}}
\int_\alpha^\infty e^{-z^2} dz \quad \text{as} \; \; m\to \infty.
\end{equation}

It is easy to see that
\begin{equation}
\label{f5} I_2 (x_m, t_m) \to 0 \quad \text{as} \;\; m \to \infty.
\end{equation}

Next we evaluate  $\lim_{m\to \infty} I_3 (x_m, t_m).$
 Since $ G (x,t;\xi,\tau) =K (x,t;\xi,\tau)- K
(-x,t;\xi,\tau),$ performing the change of variable $z = (\xi \mp
x)/2\sqrt{t}$ we obtain that $I_3 (x,t)= I_{3,1} (x,t)+I_{3,2} (x,t),
$ where
$$
I_{3,1} = \frac{1}{\sqrt{\pi}}
\int_{-\frac{x}{2\sqrt{t}}}^{\frac{b-x}{2\sqrt{t}}} e^{-z^2}
\tilde{\varphi}^\prime (x+ 2z\sqrt{t}) dz, \quad I_{3,2} =
\frac{1}{\sqrt{\pi}}
\int_{\frac{x}{2\sqrt{t}}}^{\frac{b+x}{2\sqrt{t}}} e^{-z^2}
\tilde{\varphi}^\prime (-x+ 2z\sqrt{t}) dz.
$$
Therefore, in view of (\ref{f2}), it follows that
$$
I_{3,1}(x_m, t_m) \to \tilde{\varphi}^\prime (b) \cdot
\frac{1}{\sqrt{\pi}} \int_{-\infty}^\alpha e^{-z^2} dz, \quad
I_{3,2}(x_m, t_m) \to 0 \quad \text{as} \; \; m\to \infty,
$$
which yields
\begin{equation}
\label{f6} I_3 (x_m, t_m) \to \tilde{\varphi}^\prime (b) \cdot
\frac{1}{\sqrt{\pi}} \int_{-\infty}^\alpha e^{-z^2} dz \quad
\text{as} \; \; m\to \infty.
\end{equation}

Hence, (\ref{f1}) follows from (\ref{f4})--(\ref{f6}) and (\ref{f}).

A similar argument proves that
$$
 \tilde{u}_x  (x,t) \to \tilde{\varphi}^\prime (0) \quad
\text{as} \quad (x,t) \to (0,0), \;\; (x,t) \in D_T^{\circ},
$$
which completes the proof of $(iii).$

In order to complete the proof of Lemma~\ref{lem4} it remains to show
that the condition (\ref{15}) holds, i.e., $\tilde{u} (s(t),t) = 0 $
for $ t \in [0,T).$ Since $\tilde{u} (x,t) $ satisfies
(\ref{11})--(\ref{14}) (as we proved above), by integrating the
identity (\ref{GI}) over the domain $ \varepsilon \leq \tau \leq t-
\varepsilon, \quad \delta \leq \xi \leq s(\tau) - \delta,\quad
\varepsilon, \delta >0,  $ and passing to limits, first  as $ \delta
\to 0 $ and then as $\varepsilon \to 0, $ we obtain the integral
representation (\ref{101}). Now, in view of (\ref{103}) and
(\ref{120}), it follows that
$$
-\int_0^t N_\xi (x,t;s(\tau), \tau ) g(\tau) d\tau + \int_0^t
s^\prime (\tau)N (x,t;s(\tau), \tau )  g(\tau) d\tau = 0,
$$
where $g(t)= \tilde{u} (s(t),t), \; 0 \leq t <T.$ Taking into account
that $N_\xi = - G_x $ and passing to a limit as $x\to s(t)-0,$ we
obtain by Lemma~\ref{lemF}
\begin{equation}
\label{130} \frac{g(t)}{2} + \int_0^t G_x (s(t),t;s(\tau), \tau )
 g(\tau) d\tau + \int_0^t s^\prime (\tau)N (s(t),t;s(\tau), \tau )  g(\tau)
d\tau = 0.
\end{equation}
We are going to explain that this integral equation for $g$ has only
the trivial solution $g(t ) \equiv 0.$ One can easily see that for
any $T_1 < T $ there is a constant $C>0 $ such that for $ \tau \in
[0,T_1]$
$$
 |G_x (s(t),t;s(\tau), \tau )| \leq \frac{C}{\sqrt{t-\tau}}, \quad
 |s^\prime (\tau) N (s(t),t;s(\tau), \tau )| \leq
 \frac{C}{\sqrt{t-\tau}} \,.
$$
Now, by (\ref{130}) it follows that
$$
\sup_{[0,t_1]} |g(t)|  \leq  8C\sqrt{t_1} \sup_{[0,t_1]}|g(t)|, \quad
t_1 \in (0,T_1).
$$
Choose $t_1$ so that $8C\sqrt{t_1}<1; $ then we have $g(t) = 0 $ for
$ 0\leq t \leq t_1. $

The same argument shows that if $g(t) = 0 $ for  $t\in [0,t_0],$ then
there is a $\delta>0 $ such that $g(t) = 0 $ for $t\in [0,t_0
+\delta].$ Hence, $g(t) \equiv 0  $  for $t \in [0,T), $ i.e.,
(\ref{15}) holds. This completes the proof of Lemma~\ref{lem4}.
\end{proof}

\section{Local existence-uniqueness}

We study the local existence--uniqueness properties of the system of
nonlinear integral equations (\ref{112}), (\ref{113}) by employing
the Banach Contraction Fixed Point Theorem.

Let $\varepsilon =const >0, $  and let $E$ be the space of all pairs
of continuous functions $(v(t),s^\prime(t) ),\; t \in
[0,\varepsilon].$ Equipped with the norm
$$ \|(v,s^\prime)\|_\varepsilon= \max \{\|v\|_\varepsilon, \|s^\prime\|_\varepsilon\},
\quad \text{where} \quad \|v\|_\varepsilon= \sup_{[0,\varepsilon]} |v
(t)|, \quad \|s^\prime\|_\varepsilon= \sup_{[0,\varepsilon]}
|s^\prime(t)|,$$ $E$ is a Banach space.

We fix a constant $T>0 $ and introduce the norms
$$
\|f\|_T = \sup_{[0,T]} |f(t)|, \quad \|\tilde{f}\|_T = \sup_{[0,T]}
|\tilde{f}(t)|, \quad  \|\tilde{\varphi}^\prime \|_b =\sup_{[0,b]}
|\tilde{\varphi}^\prime(x)|.
$$
In the following we may assume that $ \varepsilon < T. $

 Consider in $E$ the operator
\begin{equation}
\label{2.1} \Phi (v,s^\prime) = (A(v,s^\prime), B(v,s^\prime)),
\end{equation}
where
\begin{equation}
\label{2.2} A(v,s^\prime)(t) =2A_1(v,s^\prime)
(t)+2A_2(v,s^\prime)(t)+2A_3 (v,s^\prime)(t) \quad \text{for} \;\;
t>0
\end{equation}
with
$$A_1 (v,s^\prime) =
\int_0^t N_x (s(t),t;s(\tau),\tau)v (\tau) d\tau, \quad
 A_2 (v,s^\prime)= -\int_0^t N_x (s(t),t;0,\tau)
\tilde{f}(\tau) d\tau, $$
$$ A_3 (v,s^\prime) =\int_0^b G
(s(t),t;\xi,0) \tilde{\varphi}^\prime (\xi) d\xi,
$$
and
\begin{equation}
\label{2.3} B(v,s^\prime)(t) = \int_0^{s(t)} W(v,s^\prime) (x,t) dx,
\end{equation}
with
\begin{equation}
\label{2.4}  W(v,s^\prime)(x,t) = f(t) e^{-\lambda t} -\tilde{\sigma}
+e^{-\lambda t} \sum_{\nu =1}^3 \int_0^x J_\nu (v,s^\prime) (\xi, t)
d \xi,
\end{equation}
where $J_\nu $ are the integrals introduced in (\ref{102}).

In the above notations the system of integral equations (\ref{112})
and (\ref{113}) could be written as
\begin{equation}
\label{2.5} (v, s^\prime ) = \Phi (v, s^\prime ).
\end{equation}
Next we prove that locally (say, for $0 < t \leq \varepsilon $) the
equation (\ref{2.5}) has a unique solution by showing that for every
large enough $M>0$ there is an $\varepsilon
> 0 $ such that the operator $\Phi$ maps the closed ball ${\mathcal B}_M=
\{(v,s^\prime): \; \|(v,s^\prime)\|_\varepsilon \leq M \}$ into
itself, and its restriction on ${\mathcal B}_M$ is a contraction
mapping.

 We impose the following a priori conditions on $M$ and
$\varepsilon: $
\begin{equation}
\label{c1} M>1,  \qquad \varepsilon < \min \{1, b/(2M)\}.
\end{equation}
Then  $\|s^\prime \|_\varepsilon \leq M$ implies $|s(t) - b| \leq M
\varepsilon < b/2. $ Therefore,
\begin{equation}
\label{c2}    \|s^\prime\|_\varepsilon \leq M \;  \Rightarrow \; \;
b/2 \leq s(t) \leq  3b/2 \quad \text{for} \quad 0 \leq t \leq
\varepsilon.
\end{equation}
Straightforward computations show that the operator $\Phi $ maps
${\mathcal B}_M$ into itself if
\begin{equation}
\label{c3} M\geq  1+ (8+5b^2)\|\tilde{\varphi}^\prime \|_b + 4b^2
\|f\|_T + 4 b^2 \tilde{\sigma}
\end{equation}
and
\begin{equation}
\label{c4} \varepsilon \leq \left [ 4(M+4/b)^2 + (32/b)^2
\|\tilde{f}\|_T +  5b^2 M + 2 \|\tilde{f}\|_T \right ]^{-1}.
\end{equation}

In order to prove that the operator $\Phi : {\mathcal B}_M \to
{\mathcal B}_M $ is a contraction mapping we  investigate the
contraction properties of integral operators in
(\ref{2.2})--(\ref{2.4}). Since $A_1 (v,s^\prime) $ and $A_2
(v,s^\prime) $ are Voltera type integral operators, one can easily
prove that there exists  $\varepsilon >0$ such that for $t\in
[0,\varepsilon]$
\begin{equation} \label{c5} |A_\nu (v_1,s_1^\prime)(t)-A_\nu
(v_2,s_2^\prime)(t)|\leq \frac{1}{4} \|(v_1 - v_2, s_1^\prime -
s_2^\prime)\|_\varepsilon, \quad \nu=1,2,
\end{equation}
whenever $(v_1,s_1^\prime),\, (v_2,s_2^\prime) \in {\mathcal B}_M. $

Next we consider  $A_3 (v,s^\prime).$ Since $G(x,t;\xi,0) =
K(x,t;\xi,0)-K(-x,t;\xi,0),$ we have $ A_3 (v_1,s_1^\prime)(t)-A_3
(v_2,s_2^\prime)(t)=A_3^- (t)-A_3^+ (t), $ where
$$ A_3^\pm (t)= \frac{1}{2\sqrt{\pi t}} \int_0^b \left [
e^{-\frac{(s_1(t)\pm\xi)^2}{4t}} -e^{-\frac{(s_2(t)\pm \xi)^2}{4t}}
\right ] \tilde{\varphi}^\prime (\xi) d \xi.
$$

By the elementary inequality $|e^{-x_1} - e^{-x_2}|\leq |x_1 -x_2|
e^{ -\min(x_1,x_2)}, \; x_1,x_2 >0, $ the expression in the square
brackets in $A_3^+$  does not exceed by absolute value
$$
\frac{|s_1 (t)-s_2 (t)|}{4t} (s_1 (t)  +s_2 (t) +2 \xi) \exp  \left
(-\frac{1}{4t} \min [(s_1 (t)+\xi)^2,(s_2 (t)+\xi)^2] \right ) $$
$$\leq \frac{5b}{4} \, e^{-b^2/16t}  \cdot \|s_1^\prime -s_2^\prime\|_\varepsilon
$$
because $b/2 \leq s_i (t) \leq 3b/2, \; i=1,2 $ (see (\ref{c2})).
Therefore, taking into account that $e^{-b^2/16t} < 16t/b^2,$ we
obtain
$$
|A_3^+ (t) | \leq 10 \sqrt{t} \|\tilde{\varphi}^\prime\|_b
\|s_1^\prime -s_2^\prime\|_\varepsilon \leq \frac{1}{8} \|s_1^\prime
-s_2^\prime\|_\varepsilon \quad \text{if} \quad t\leq \varepsilon
\leq (80\|\tilde{\varphi}^\prime\|_b)^{-2}.
$$

In order to estimate $A_3^-$ we write it in the form $A_3^- (t)
=A_{3,1}^- (t) +A_{3,2}^- (t),$  where
$$A_{3,1}^-(t)=
\frac{1}{2\sqrt{\pi t}} \int_0^{b-\delta}
e^{-\frac{(s_1(t)-\xi)^2}{4t}} \left [1-\exp \left (
\frac{(s_1(t)-\xi)^2}{4t}-\frac{(s_2(t)-\xi)^2}{4t} \right ) \right ]
\tilde{\varphi}^\prime (\xi) d \xi,
$$
$$A_{3,2}^- (t)=
\frac{1}{2\sqrt{\pi t}} \int_{b-\delta}^b
e^{-\frac{(s_1(t)-\xi)^2}{4t}} \left [1-\exp \left (
\frac{(s_1(t)-\xi)^2}{4t}-\frac{(s_2(t)-\xi)^2}{4t} \right ) \right ]
\tilde{\varphi}^\prime (\xi) d \xi.
$$
and $  \delta = 2Mt^{1/4}.$

From (\ref{c2}) it follows that $|s_1 (t) + s_2 (t) - 2 \xi | \leq 3b
$ which implies
$$
\left |\frac{(s_1(t)-\xi)^2}{4t}-\frac{(s_2(t)-\xi)^2}{4t} \right | =
\frac{|s_1 (t)-s_2 (t)|}{4t}|s_1 (t) + s_2 (t) - 2 \xi | \leq
\frac{3b}{4} \cdot \|s_1^\prime - s_2^\prime\|_\varepsilon.
$$
Therefore, by the inequality $|e^x -1|\leq |x|e^{|x|},$ the
expression in the square brackets in $A_{3,1}^-(t)$ does not exceed
by absolute value $ (3b/4) \exp (3bM/2)\|s_1^\prime -
s_2^\prime\|_\varepsilon.$ Thus, it follows
$$
\left |A_{3,1}^-(t) \right | \leq \frac{3b}{4} \, e^{3bM/2}\,
\|\tilde{\varphi}^\prime \|_b \cdot  \|s_1^\prime -
s_2^\prime\|_\varepsilon \cdot \int_0^{b-\delta}
\frac{1}{2\sqrt{t}}\, e^{-\frac{(s_1(t)-\xi)^2}{4t}} d\xi.
$$

Performing the change of variable $z=\frac{\xi - s_1(t)}{2\sqrt{t}} $
in the latter integral, and  estimating from above the resulting
integral, we obtain
$$
\int_{-\frac{s_1 (t)}{2\sqrt{t}}}^{\frac{b-\delta-s_1
(t)}{2\sqrt{t}}} e^{-z^2} dz \leq \frac{b-\delta}{2\sqrt{t}}
e^{-\frac{(b-\delta-s_1 (t))^2}{4t}}\leq \frac{b}{2\sqrt{t}}
e^{-\frac{M^2}{4\sqrt{t}}} \leq 8b\cdot t^{1/2}
$$
because $$ \delta + (s_1 (t)-b) \geq 2M t^{1/4} -
\|s_1^\prime\|_\varepsilon \cdot t \geq 2M t^{1/4} - Mt \geq M
t^{1/4}, $$ and (by (\ref{e3}) and (\ref{c1}))  $\; t^{-1} \exp
(-M^2/4\sqrt{t}) \leq 16.$  Therefore,
$$
\left |A_{3,1}^-(t) \right | \leq 6b^2 \exp (3bM/2)
\|\tilde{\varphi}^\prime \|_b \cdot \|s_1^\prime -
s_2^\prime\|_\varepsilon \cdot \sqrt{t}.
$$

Next we estimate $A_{3,2}^-(t).$  If $\xi \in [b-\delta, b],$ then
(since $s_1(0)=s_2 (0)=b$)
$$
|s_1 (t) + s_2 (t) - 2 \xi| = |s_1(t)-s_1 (0) + s_2 (t)- s_2 (0) +
2(b-\xi)|
$$ $$
\leq (\|s_1^\prime \|_\varepsilon
 + \|s_2^\prime \|_\varepsilon)\cdot t + 2 \delta \leq 2M t + 4M t^{1/4}
\leq 6M t^{1/4},
$$
which implies
$$
\left |\frac{(s_1(t)-\xi)^2}{4t}-\frac{(s_2(t)-\xi)^2}{4t} \right | =
\frac{|s_1 (t)-s_2 (t)|}{4t}|s_1 (t) + s_2 (t) - 2 \xi | \leq
\frac{3}{2}M t^{1/4} \|s_1^\prime - s_2^\prime\|_\varepsilon.$$
Estimating  the expression in the square brackets in $A_{3,2}^-(t)$
as in the case of $A_{3,1}^-(t) $ and taking into account (\ref{e2}),
we obtain
$$
\left |A_{3,2}^-(t) \right | \leq \frac{3}{2} M t^{1/4}\exp
(3M^2)\|\tilde{\varphi}^\prime \|_b \|s_1^\prime -
s_2^\prime\|_\varepsilon.
$$
Thus,
$$
|A_3^- (t) | \leq \left [6b^2 \exp (3bM/2) + 3M\exp (3M^2) \right ]
\|\tilde{\varphi}^\prime \|_b \|s_1^\prime - s_2^\prime\|_\varepsilon
\cdot t^{1/4},
$$
which implies that
$$
|A_3^- (t) | \leq \frac{1}{8} \|s_1^\prime - s_2^\prime\|_\varepsilon
$$
$ \text{if} \qquad \qquad  t \leq \varepsilon \leq 8^{-4} \left [6b^2
\exp (3bM/2) + 3M\exp (3M^2) \right ]^{-4} \|\tilde{\varphi}^\prime
\|_b^{-4}. $\vspace{2mm}

Now, the estimates for $A_3^-(t)$ and $A_3^+ (t)$ imply that if
\begin{equation}
\label{c6} \varepsilon \leq \min \left \{80\|\tilde{\varphi}^\prime
\|_b)^{-2},\, 8^{-4} \left [6b^2 \exp (3bM/2) + 3M\exp (3M^2) \right
]^{-4} \|\tilde{\varphi}^\prime \|_b^{-4} \right \},
\end{equation}
then
\begin{equation}
\label{c7} |A_3 (v_1,s_1^\prime)(t)-A_3 (v_2,s_2^\prime)(t)| \leq
\frac{1}{4} \|s_1^\prime - s_2^\prime\|_\varepsilon.
\end{equation}

The estimates (\ref{c5}) hold if $\varepsilon $ satisfies
inequalities similar to (\ref{c4}) and (\ref{c6}). Moreover, one can
prove that the operator $B$ is a contraction in $\mathcal{B}_M $ if
$\varepsilon $ satisfies similar restrictions. We omit the details,
but it is important to note that the right-hand sides of (\ref{c4}),
(\ref{c6}) and the analogous inequalities (which guarantee that the
operator $\Phi$ is a contraction  on ${\mathcal B}_M$ with a
contraction coefficient $<1$) are given by expressions that decrease
if the parameters involved (such as $b, M, \|\tilde{\varphi}^\prime
\|_b, \|f\|_T, \|\tilde{f}\|_T $) increase.

Therefore, applying the Banach Contraction Fixed Point Theorem, we
obtain the following statement.

\begin{Lemma}
\label{lem5} (a) For each constant $M>1 $ which satisfies (\ref{c3})
there is a constant $\varepsilon >0 $ such that the system of
integral equations (\ref{112}) and (\ref{113}) (with $s(t) =b+
\int_0^t s^\prime (\tau ) d\tau$) has a unique solution $(v(t), \,
s^\prime (t)), \;  t\in [0,\varepsilon],$ such that $
\|v\|_\varepsilon \leq M$ and $\|s^\prime \|_\varepsilon \leq M.$

(b) The constant $\varepsilon $ may be chosen so that
\begin{equation}
\label{c14} \varepsilon =h(M,b, 1/b, \|f\|_T, \|\tilde{f}\|_T,
\|\tilde{\varphi}^\prime\|_b),
\end{equation}
where $ h(y_1,y_2,y_3,y_4,y_5,y_6 ), \; y_i >0, $  is a monotone
decreasing function with respect to each argument $y_i,
\,i=1,\ldots,6. $
\end{Lemma}

Next we prove uniqueness of solutions of Problem $\tilde{P}.$
\begin{Lemma}
\label{lem6} For each $ T\leq \infty, $ Problem~$\tilde{P} $ has at
most one solution for $t \in [0,T).$
\end{Lemma}

\begin{proof}
Suppose that $(\tilde{u}_1 (x,t), s_1 (t)) $ and $(\tilde{u}_2 (x,t),
s_2 (t)) $ are two solutions of Problem~$\tilde{P}$  on the interval
$[0,T),\; T\leq \infty.$ Then, in view of Lemma~\ref{lem4}, the pairs
of functions $(v_1 (t), s_1^\prime (t))$ and $(v_2 (t), s_2^\prime
(t)),$ where $v_1 (t) =\frac{\partial \tilde{u}_1}{\partial x} (s_1
(t),t),$ $  v_2 (t) =\frac{\partial \tilde{u}_2}{\partial x} (s_2
(t),t), $ are solutions of the system of integral equations
(\ref{112}), (\ref{113}). Fix $\varepsilon_0<T$ and choose $M>1 $ so
that $$M \geq \max \{\|v_1 (t)\|_{\varepsilon_0},
\|s_1^\prime\|_{\varepsilon_0}, \|v_2 (t)\|_{\varepsilon_0},
\|s_2^\prime\|_{\varepsilon_0} \}.$$ By Lemma~\ref{lem5}, there is a
positive constant $\varepsilon < \varepsilon_0$ such that the pairs
$(v_1 (t), s^\prime_1 (t))$ and $(v_2 (t), s^\prime_2 (t)) $ coincide
on the interval $[0,\varepsilon].$ Therefore, the integral
representation (\ref{111}) implies  $$ \tilde{u}_1 (x,t) =\tilde{u}_2
(x,t), \quad s_1 (t) = s_2 (t) \quad \text{if} \quad 0 \leq t \leq
\varepsilon, \; 0 \leq x \leq s(t).$$

Having proved uniqueness for a small time  interval $0 \leq t \leq
\varepsilon, $ we can proceed in a similar way, step by step, to get
uniqueness for all $t>0.$  Let $t_0 <T $ be a positive number such
that
\begin{equation}
\label{2.120}
 \tilde{u}_1 (x,t) =\tilde{u}_2 (x,t), \quad s_1 (t) = s_2 (t) \quad \text{if}
\quad 0 \leq t \leq t_0, \; 0 \leq x \leq s(t).
\end{equation}
Then $(e^{-\lambda t_0} \tilde{u}_1 (x,t+t_0), s_1 (t+t_0)) $ and
$(e^{-\lambda t_0} \tilde{u}_2 (x,t+t_0), s_2 (t+t_0)) $ are two
solutions of Problem~$\tilde{P}$ on the interval $[0,T-t_0), $  if
considered with $f_1 (t) = f(t+t_0),$
$$ \tilde{f}_1 (t)=e^{-\lambda t_0} \tilde{f}(t+t_0), \;
 \tilde{\varphi}_1 (x) =e^{-\lambda t_0} \tilde{u}_1
 (x,t_0)=e^{-\lambda t_0} \tilde{u}_2 (x,t_0), \;
 b_1 = s_1 (t_0)= s_2 (t_0) $$ instead of $f(t), \tilde{f}(t),
\tilde{\varphi} (x) $ and $b.$ By the above argument, there is a
constant $\varepsilon_1 >0$ such that $$ \tilde{u}_1 (x,t)
=\tilde{u}_2 (x,t), \quad s_1 (t) = s_2 (t) \quad \text{if} \quad 0
\leq t \leq t_0 +\varepsilon_1, \; 0 \leq x \leq s(t). $$ Therefore,
(\ref{2.120}) holds for each $t_0 <T, $ i.e. the solutions
$(\tilde{u}_1 (x,t), s_1 (t)) $ and $(\tilde{u}_2 (x,t), s_2 (t)) $
coincide on $[0,T).$
\end{proof}

\section{Existence of global solution}

Lemma \ref{lem4} and Lemma \ref{lem5}  guarantee that the
Problem~$\tilde{P}$ has a solution for $0\leq t < \varepsilon $ for
sufficiently small $\varepsilon
>0.$ In order to prove the existence of a global solution we need
 a priori estimates for $s(t) $ and  $ \tilde{u}_x (x,t). $

By Lemma \ref{lem2}, there are constants $C_1 =C_1 (T) $ and $C_2
=C_2 (T)$ such that
\begin{equation}
\label{221}  |s^\prime (t)| \leq C_1,  \quad  1/C_2 \leq  |s (t)|
\leq C_2, \quad 0 \leq t <T.
\end{equation}

\begin{Lemma}
\label{lem7} Suppose that the pair of functions $(\tilde{u} (x,t),
s(t)) $ is a solution of Problem~$\tilde{P} $ for $ t \in [0,T), \;
0<T<\infty.$ Then
\begin{equation}
\label{220}  \Psi_T := \sup_{D_T} |\tilde{u}_x (x,t)| <\infty,
\end{equation}
where $D_T =\{(x,t): \; 0 \leq x \leq s(t), \; 0\leq t<T\}. $
 \end{Lemma}

\begin{proof}
It is enough to prove that
\begin{equation}
\label{222}  m: = \sup_{[0,T)} |v (t)| < \infty,
\end{equation}
where $v(t) =\tilde{u}_x (s(t),t).$ Indeed, by (\ref{f01}),
$\tilde{u}_x (x,t) = \sum_{\nu=1}^3 I_\nu (x,t), $ where the
integrals $I_\nu (x,t)$ are given by (\ref{f02}).

We have $ I_1 (x,t) = I_{1,1} (x,t)+ I_{1,2} (x,t), $ where $I_{1,1}
(x,t)$ and $ I_{1,2} (x,t) $ are given in (\ref{f11}) and
(\ref{f12}). First we estimate $|I_{1,1} (x,t)|:$
$$ |I_{1,1} (x,t)| \leq  m \cdot \left ( \frac{1}{\sqrt{\pi}} \int_0^t
\frac{|s(t) -s(\tau)|}{4(t-\tau)^{3/2}}
e^{-\frac{(s(\tau)-x)^2}{4(t-\tau)}} d\tau  \right. $$
$$ + \left.
 \frac{1}{\sqrt{\pi}} \int_0^t \frac{|s(t)-x|}{4(t-\tau)^{3/2}}
e^{-\frac{(s(t)-x)^2}{4(t-\tau)}} \exp \left [ \frac{(s(t)-x)^2
-(s(\tau)-x)^2}{4(t-\tau)}  \right ] d\tau \right ).$$ From
(\ref{221}) it follows that $|s(t) -s(\tau)|/(t-\tau) \leq C_1,$ so
the first integral in the brackets does not exceed $  \int_0^t
C_1/(2\sqrt{t-\tau}) d\tau \leq C_1 \sqrt{T}. $ By (\ref{221}), the
expression in the square brackets in the integrand of the second
integral can be estimated from above by
$$
\frac{|s(t)-s(\tau)|}{4(t-\tau)} |s(t) + s(\tau) - 2x | \leq C_1 C_2.
$$
Therefore, in view of (\ref{e1}), the second integral does not exceed
$e^{C_1C_2}.$ Hence,
$$
|I_{1,1} (x,t)| \leq  m \cdot \left (C_1 \sqrt{T} + e^{C_1C_2} \right
).
$$

By (\ref{221}) we have $ C^{-1}_2 \leq x+s(\tau) \leq 2C_2 $ for $x <
s(t).$ From these inequalities and (\ref{e1}) it follows
$$ |I_{1,2} (x,t)| \leq m \cdot 2 C_2^2 \int_0^t  K_x
(C^{-1}_2,t;0,\tau) d\tau \leq m \cdot 2 C_2^2. $$

On the other hand (\ref{e1}) and (\ref{e2}) imply
\begin{equation}
\label{224} |I_2 (x,t) | \leq \|\tilde{f}\|_T, \quad  |I_3 (x,t) |
\leq 2 \|\tilde{\varphi}^\prime \|_b.
\end{equation}

Hence,
$$
\sup_{D_T} |\tilde{u}_x (x,t)| \leq m \cdot \left (C_1 \sqrt{T} +
e^{C_1C_2} + 2C_2^2 \right ) + \|\tilde{f}\|_T +
2\|\tilde{\varphi}^\prime \|_b,
$$
i.e.,  (\ref{222}) implies (\ref{220}).

Next we prove (\ref{222}). By (\ref{112}), $ v(t) = 2\sum_{\nu=1}^3
I_\nu (s(t),t), $ where the integrals $I_\nu (x,t)$ are given by
(\ref{f02}).

First we consider $I_1 (s(t),t) = \int_0^t N_x (s(t),t;s(\tau),\tau)
v(\tau) d\tau. $ Since $ N (x,t;\xi,\tau)=  K (x,t;\xi,\tau)+ K
(-x,t;\xi,\tau),$ we have
$$ |N_x (s(t),t;s(\tau),\tau)|\leq
\frac{|s(t)-s(\tau)|}{4(t-\tau)^{3/2}}e^{-\frac{(s(t)-s(\tau))^2}{4(t-\tau)}}+
\frac{|s(t) +s(\tau)|}{4(t-\tau)^{3/2}} e^{-\frac{(s(t)
+s(\tau))^2}{4(t-\tau)}}.$$

From (\ref{221}) it follows that $|s(t) -s(\tau)|/(t-\tau)\leq C_1,$
so the first term on the right in the above inequality is less than $
C_1/(4\sqrt{t-\tau}). $ On the other hand,  (\ref{221}) implies
$2/C_2 \leq s(t) +s(\tau) \leq 2 C_2. $ Therefore, in view of
(\ref{e3}),
$$
\frac{|s(t) +s(\tau)|}{4(t-\tau)^{3/2}} e^{-\frac{(s(t)
+s(\tau))^2}{4(t-\tau)}} \leq  \frac{2C_2}{4(t-\tau)^{3/2}}
e^{-\frac{1}{C_2^2(t-\tau)}}   \leq C_2^3 \frac{1}{2\sqrt{t-\tau}}
\,.
$$
Thus, we obtain
\begin{equation}
\label{300}
 |N_x (s(t),t;s(\tau),\tau)|\leq
(C_1/2 + C_2^3) \frac{1}{2\sqrt{t-\tau}}, \quad 0 \leq \tau < t <T.
\end{equation}

Let  $\delta \in (0,T)$  (later we will choose $\delta$ sufficiently
small), and let
$$
\mu (t) = \sup \{|v(\tau)|: \; 0 \leq \tau \leq  t \}.
$$
For $t\in (T-\delta, T), $   (\ref{300}) implies
$$
|I_1 (s(t),t)| \leq   \mu (T-\delta) \, \int_0^{ T-\delta} |N_x
(s(t),t; s(\tau), \tau)| d \tau + \mu (t) \,  \int_{ T-\delta}^t |N_x
(s(t),t; s(\tau), \tau)| d \tau
$$
$$
\leq \mu (T-\delta) \cdot  (C_1/2 + C_2^3) \sqrt{T}   + \mu (t) \cdot
(C_1/2 + C_2^3) \sqrt{\delta}.
$$
Therefore, in view of (\ref{112}) and (\ref{224}), we obtain
\begin{equation}
\label{d0} |v(t)| \leq \mu (T-\delta) \cdot (C_1 + 2C_2^3) \sqrt{T} +
\mu (t) \cdot (C_1 + 2C_2^3) \sqrt{\delta} + 2\|\tilde{f}\|_T+
4\|\tilde{\varphi}^\prime \|_b
\end{equation}
for $ t \in (T-\delta, T).$

Choose $\delta $ so that
$$
 (C_1 + 2C_2^3) \sqrt{\delta} <1/2.
$$
Then (\ref{d0}) implies
$$
\mu (t) \leq  2\mu (T-\delta) (C_1 + 2C_2^3) \sqrt{T} +
4\|\tilde{f}\|_T+ 8\|\tilde{\varphi}^\prime \|_b \quad   \text{for}
\;\; t\in (T-\delta, T).
$$
Hence, $ m = \sup \{ \mu (t), \;t \in [0,T) \}  <\infty, $ i.e.,
(\ref{222}) holds, which completes the proof of Lemma~\ref{lem7}.
\end{proof}

{\em Proof of Theorem \ref{thm2}.} By Lemma \ref{lem6},
Problem~$\tilde{P}$ has at most one global solution.  Now we prove
that Problem $\tilde{P}$ has a global solution. Assume the contrary,
and let $T$ be the greatest positive number such that
Problem~$\tilde{P}$ has a solution for $t \in [0,T).$ Let the pair of
functions $\tilde{u} (x,t) $ and $s(t) $ be a solution of
Problem~$\tilde{P}$ for $t \in [0,T).$

For each $t_0 < T $ we can consider a ``modified`` Problem
$\tilde{P}$
 with data
\begin{equation}
\label{322} f_1 (t)= f(t_0+t), \;\;  \tilde{f}_1 (t)=e^{-\lambda t_0}
\tilde{f}(t_0+t), \;\;
 \tilde{\varphi}_1 (x) =e^{-\lambda t_0} \tilde{u}
 (x,t_0), \;\;
 b_1 = s (t_0)
\end{equation}
instead of $f(t), \tilde{f}(t), \tilde{\varphi} (x) $ and $b.$ By the
local existence--uniqueness result given in Lemma~\ref{lem5}, for
each $M_1 > 1 $ which satisfies
\begin{equation}
\label{2c2}  M_1 \geq  1+ (8+5b_1^2)\|\tilde{\varphi}_1^\prime
\|_{b_1} + 4b_1^2 \|f_1\|_T + 4 b_1^2 \tilde{\sigma}
\end{equation}
and each $\varepsilon >0$ with
\begin{equation}
\label{2c14} \varepsilon \leq h(M_1, b_1, 1/b_1, \|f_1\|_T,
\|\tilde{f}_1\|_T, \|\tilde{\varphi}_1^\prime\|_{b_1}),
\end{equation}
the ``modified`` Problem $\tilde{P}$ has a solution $(\tilde{u}_1
(x,t), s_1 (t) )$ for $0\leq t <\varepsilon.$  Then, the pair
$(\tilde{U} (x,t), S(t))$  with
$$  \tilde{U} (x,t) =\begin{cases} \tilde{u} (x,t) & 0\leq t \leq t_0 \\
e^{\lambda t_0} \tilde{u}_1 (x,t-t_0) & t_0 \leq t < t_0 +
\varepsilon
\end{cases}, \quad S(t)=
\begin{cases} s(t)& 0\leq t \leq t_0\\
s_1 (t-t_0) & t_0 \leq t < t_0 + \varepsilon
\end{cases}
$$
is a solution of Problem $\tilde{P}$ for $0\leq t <t_0 +
\varepsilon.$

 Moreover, in view of the a priori estimates given in
Lemma~\ref{lem2} and Lemma~\ref{lem7}, by Lemma~\ref{lem5} we can
choose {\em one and the same} $\varepsilon$ for every $ t_0 <T. $
Indeed, let us set
$$  \tilde{M} =1+ (8+5C_2^2) \Psi_T + 4C_2^2 \|f\|_{2T} + 4 C_2^2
\tilde{\sigma}, \quad \tilde{\varepsilon} = h(\tilde{M}, C_2, C_2,
\|f\|_{2T}, \|\tilde{f}\|_{2T}, \Psi_T).
$$

Therefore, choosing $t_0> T-\tilde{\varepsilon}, $ we get the
existence of a solution of Problem~$\tilde{P}$ for $ t \in [0, t_0
+\tilde{\varepsilon}) $ with $ t_0 +\tilde{\varepsilon} >T, $ which
contradicts the choice of $T.$ Hence Problem~$\tilde{P}$ has a global
solution for $t \in [0, \infty), $ i.e., Theorem~\ref{thm2} holds.

In view of Lemma~\ref{lem3}, this implies that Problem~$P$ has a
unique global solution, i.e., Theorem~\ref{thm1} holds as well.

\section{Concluding remarks}

1. During the last 40 years various mathematical models for evolution
of tumors have been developed and analyzed --  see the survey papers
\cite{BLM08, Fr07} and the bibliography therein. Some of those models
are in the form of free boundary problems for partial differential
equations, whereby the tumor surface is a free boundary and the tumor
growth is determined by the level of a diffusing nutrient
concentration \cite{G72,G76}  (see also \cite{BLM08}--\cite{CF02},
\cite{FR99}). The main physical and biological concepts underlying
such type of models are the mass conservation law and
reaction--diffusion processes within the tumor. Usually additional
geometric assumptions on the shape of the tumor are imposed -- see,
for instance, \cite{G72,FR99}, where the tumor is supposed to be
spherically symmetric.

A slight modification of those  models is considered in \cite{R1}. It
describes the growth of an avascular solid tumor which receives
nutrient supply via a diffusion process only through some part of its
boundary  (called {\em base} of the tumor), and it is assumed that
there is no nutrient flow through the remaining part of the boundary.
Moreover, the tumor is supposed to be thin and approximately
disc-shaped, so only one spatial dimension, say $x,$ is considered.
With tumor's base situated at $x=0$ the nutrient concentration
$\sigma (x,t)$ satisfies the reaction--diffusion equation
\begin{equation}
\label{0.1} c \frac{\partial \sigma }{\partial t} = \frac{\partial^2
\sigma}{\partial x^2} - \lambda \sigma, \quad 0 < x < s(t), \;\; t>0,
\end{equation}
where $s(t)>0 $ is the tumor's thickness at time $t,$  $ \lambda
=const >0, $ $\lambda \sigma$ is the nutrient consumption rate, and
$c>0$ is a dimensionless constant coming as a ratio of the nutrient
diffusion time scale to the tumor growth time scale.

Following \cite{G72}, it is assumed  that all tumor cells are
physically identical in volume and mass, and that the cell density is
constant throughout the tumor. As in  \cite{FR99}, the cell
proliferation rate within the tumor is given by $ P(\sigma )= \mu (
\sigma - \tilde{\sigma}),$ where $\mu $ and $\tilde{\sigma}$ are
positive constants. These assumptions lead to the equation
\begin{equation}
\label{0.2}  s^\prime (t) = \mu \int_0^{s(t)} (\sigma (x,t) -
\tilde{\sigma}) dx, \quad     t\geq 0.
\end{equation}
In addition, the following {\em mixed type} boundary conditions hold:
\begin{equation}
\label{0.3} \sigma (0, t) = f(t), \quad   \;\; t\geq 0,
\end{equation}
\begin{equation}
\label{0.4}  \sigma(x,0) = \varphi (x),  \quad  x\in [0,b], \;\;
s(0)=b >0, \quad  \varphi (0) = f(0),
\end{equation}
\begin{equation}
\label{0.5} \frac{\partial \sigma}{\partial x} (s(t), t) = 0, \quad
t>0,
\end{equation}
where $f(t)>0 $ is the external nutrient concentration at the base of
the tumor at time $t$,   $\varphi (x)>0 $ is the initial nutrient
concentration within the tumor, and the condition (\ref{0.5}) comes
because it is assumed that there is no  nutrient transfer through the
free boundary $x= s(t). $

The parameters $ c,  \mu $ and $\lambda $  in
(\ref{0.1})--(\ref{0.5}) depend on the choice of time and length
units. One may scale out $x$ and $t$ in an appropriate way in order
to get $c=1$ and $ \mu= 1 $ ($\lambda $ may change as well), which
shows that Problem (\ref{0.1})--(\ref{0.5}) is equivalent to
Problem~$P.$\vspace{2mm}

2. Another interesting question in the study of mathematical models
of tumor growth is under what conditions does the tumor grow, shrink
or become dormant. In order to answer that question one needs to find
the stationary solution (which gives the dormant case) and analyze
its asymptotic stability (see \cite{FR99,CF01,CF02} and the
bibliography there).

In the case of Problem~$P,$ if $f(t) = \bar{\sigma}= const $ then it
is easy to see that the stationary solution is given by the pair
$(\bar{u}(x), \bar{b}), $ where
\begin{equation}
\label{r1} \bar{u}(x) = \bar{\sigma}\, \frac{\cosh \left (
\sqrt{\lambda} (x-\bar{b})\right )}{\cosh \left ( \sqrt{\lambda}
\,\bar{b} \right )}
\end{equation}
and $\bar{b} $ is determined by the equation
\begin{equation}
\label{r2} \bar{\sigma} \, \tanh \left (\bar{b}\sqrt{\lambda} \right
) = \tilde{\sigma}\bar{b}\, \sqrt{\lambda} \,.
\end{equation}
\vspace{2mm}

3. In (\ref{6}) of Theorem \ref{thm1}, the assumptions $f(t)>0 $ for
$t \in [0, \infty) $ and $\varphi (x) >0 $ for $x\in [0,b] $ come
from the corresponding mathematical model (Section 7.1). However, the
result stated in  Theorem~\ref{thm1} remains valid without those
requirements.

Indeed, let $f(t) \in C^1 ([0,\infty))$   and $ \varphi (x) \in C^2
([0,b]) $ be arbitrary functions. Then, under the assumptions of
Lemma~\ref{lem2}, the following a priori  estimates hold:
\begin{equation}
\label{c1.7} | u(x,t)| \leq C_T, \quad  0 \leq x \leq s(t), \;\; 0
\leq t <T,
\end{equation}
\begin{equation}
\label{c1.8} -(C_T + \tilde{ \sigma}) s(t) \leq s^\prime (t) \leq
(C_T-\tilde{\sigma})s(t), \quad be^{-(C_T +\tilde{\sigma})t} \leq
s(t) \leq be^{(C_T-\tilde{\sigma})t},
\end{equation}
where $b=s (0), \;   C_T = \max \left ( \sup_{[0,T)} |f(t)|, \;
\sup_{[0,b]} |\varphi (x)| \right ).$ The proof of Theorem~\ref{thm1}
is the same, but one needs to use the estimates  (\ref{c1.7}) and
(\ref{c1.8}) instead of (\ref{1.7}) and (\ref{1.8}) in
Lemma~\ref{lem2}.\vspace{2mm}

4. It is known that the free boundary in the one-dimensional Stefan
problem (see \cite[Ch. 8]{F}, \cite{Rub}, \cite[Ch. 17]{C84}) is a
$C^\infty$-curve (see \cite{CH68,CP71,Sch} and the bibliography
therein). In the context of tumor models a similar result is proven
in \cite[Theorem 4.1]{CF02}.

In the case of our Problem $P$ it is easy to see that $s(t)\in C^2
([0, \infty)).$  Indeed, since $u_t (x,t)$ is defined and continuous
for $ 0 \leq x \leq s(t), \, t>0, $ from (\ref{2}) it follows
$$
s^{\prime \prime}(t) = [u(s(t),t) -\tilde{\sigma}] s^\prime (t)+
\int_0^{s(t)} u_t (x,t) dx.
$$
Therefore, taking into account that
$$
\int_0^{s(t)} u_t (x,t) dx= \int_0^{s(t)} ( u_{xx}- \lambda u) dx =
u_x (s(t),t) -u_x (0,t) - \lambda \int_0^{s(t)} u(x,t) dx
$$
we obtain, in view of (\ref{2}) and (\ref{5}),
$$
s^{\prime \prime}(t) = [u(s(t),t) -\tilde{\sigma}- \lambda ] s^\prime
(t) - u_x (0,t) -\lambda \, \tilde{\sigma} s(t),
$$
where the expression on the right is a continuous function for $t\geq
0,$ i.e., $s(t) \in C^2 ([0,\infty)).$

 However, higher derivatives of $s(t) $ may not exist if we assume
 $f \in C^1 ([0,\infty))$ only, since
the condition (\ref{2}) in Problem~$P$ is nonlocal (compare with the
case of one-dimensional Stefan problem, where the infinite
differentiability of the free boundary does not require infinite
differentiability of the boundary data at $x=0$ -- see \cite{Sch}).
In our case, one can prove the following: {\em In Problem~P, the free
boundary $x=s(t), \; t \in (0,\infty)$ is an infinitely
differentiable curve if and only if} $f(t) \in C^\infty
((0,\infty)).$ We will present the details somewhere else.

\end{document}